\documentclass{aptpub}

\usepackage{graphicx}

\authornames{G. Vasdekis and G. O. Roberts} 
\shorttitle{A Note on the Polynomial Ergodicity of the One-Dimensional Zig-Zag process} 

\DeclareMathOperator{\sgn}{sgn}

\begin{document}

\title{A Note on the Polynomial Ergodicity of the One-Dimensional Zig-Zag process} 

\authorone[University of Warwick]{G. Vasdekis}
\authortwo[University of Warwick]{G. O. Roberts} 

\addressone{Department of Statistics,
University of Warwick,
Coventry,
CV4 7AL} 
\emailone{Giorgos.Vasdekis.1@warwick.ac.uk} 
\addresstwo{Department of Statistics,
University of Warwick,
Coventry,
CV4 7AL}
\emailtwo{Gareth.O.Roberts@warwick.ac.uk}
\begin{abstract}
We prove polynomial ergodicity for the one-dimensional Zig-Zag process on heavy tailed targets and identify the exact order of polynomial convergence of the process when targeting Student distributions.
\end{abstract}



\section{Introduction}

 The Zig-Zag process is a Piecewise Deterministic Markov Process (PDMP) that was recently used as a new way to construct MCMC algorithms. The one dimensional Zig-Zag appeared in \cite{bierkens.roberts_scaling:2017} as a scaling limit of the Lifted Metropolis-Hastings (see \cite{turitsun.chertkov.vucelja:11,diaconis.holmes.neal:00}) applied to the Curie-Weiss model (see \cite{levin.luczak.peres:07}). 
 The process was later extended in higher dimensions in \cite{bierkens.roberts_superefficient:2019} and has been proposed as a way to sample from posterior distributions in a Bayesian setting (see also \cite{fearnhead.bierkens.pollock.roberts:18, vanetti.cote.deligiannidis.doucet:17}). Since then, its properties have been extensively studied in the literature (see for example \cite{bierkens.duncan:17, bierkens.kamatani.roberts:18, bierkens.lunel:19, bierkens.nyquist.schlottke:19} etc.).
 
\cite{bierkens.roberts.zitt:2019} proves ergodicity and exponential ergodicity of the Zig-Zag process in arbitrary dimension, however a crucial assumption required for exponential ergodicity  in that work is that the target density has exponential or lighter tails. In \cite{vasdekis.roberts.suzz:21} the converse result was proven. The Zig-Zag sampler fails to be exponentially ergodic when the target distribution is heavy tailed.
On the other hand, it was shown in Theorem 1 of \cite{bierkens.roberts.zitt:2019} that the process will converge to the invariant measure under very mild assumptions, including the heavy tailed case. Furthermore,  \cite{andrieu.dobson.wang:21} used hypocoercivity techniques (see also \cite{andrieu.durmus.nusken.roussel:18}) to prove polynomial rates of convergence for the Zig-Zag process on heavy tailed targets in arbitrary dimension.

In this note we will focus on the one-dimensional Zig-Zag process and prove polynomial ergodicity in the heavy tailed scenario. The result applies in the special case where the tails of the target decay in the same manner as a Student distribution with $\nu$ degrees of freedom. In that case, we prove that the polynomial rate of convergence is arbitrarily close to $\nu$, but not more than $\nu$. This improves upon the result stated in \cite{andrieu.dobson.wang:21} in the special case where $d=1$, although their work provides convergence results for higher dimensions as well. 

The rest of this paper is organised as follows. In Section 2 we recall the definition of the one-dimensional Zig-Zag process and we state the main result concerning its polynomial ergodicity in heavy tailed targets. In Section 3 we provide proof of this result. Finally, in Section 4 we discuss the rates of polynomial convergence of the Zig-Zag process and compare them with the ones of other state of the art Metropolis-Hastings algorithms.

\section{Results}
We begin with recalling the definition of the one-dimensional Zig-Zag process. Let $E = \mathbb{R} \times \{ -1,+1 \}$, $U \in C^1(E)$ and $\lambda:E \rightarrow \mathbb{R}_{\geq 0}$ with 
\begin{equation}
    \lambda(x,\theta)=[\theta U'(x)]^++\gamma(x),
\end{equation}
where $\gamma$ is a non-negative integrable function and we write $a^+=\max\{ a,0 \}$. The one-dimensional Zig-Zag process $(Z_t)_{t \geq 0}=(X_t,\Theta)_{t \geq 0}$ is a continuous time Markov process with state space $E$ which evolves as follows. If the process starts from $(x,\theta) \in E$, then $X_t=x+t\theta$ and $\Theta_t=\theta$ for all $t < T_1$, where $T_1$ is the first arrival time of a non-homogeneous Poisson process with rate $m_1(s)=\lambda(x+s\theta,\theta)$. Then $X_{T_1}=x+T_1\theta$, $\Theta_{T_1}=-\theta$. Then, $X_{t}=X_{T_1}+(t-T_1)\Theta_{T_1}$ and $\Theta_t=\Theta_{T_1}$ for all $t \in (T_1,T_1+T_2)$, where $T_2$ is the first arrival time of a Poisson process with intensity $m_2(s)=\lambda(X_{T_1}+s\Theta_{T_1},\Theta_{T_1})$. The process is then defined inductively up to time $T_n$ for all $n \in \mathbb{R}$.

It is proven in \cite{bierkens.roberts_superefficient:2019} that with this choice of $\lambda$, the process has measure $\mu$ invariant, where
\begin{equation}\label{def.mu:1}
    \mu=\pi \bigotimes \frac{1}{2}\delta_{\{-1,+1\}}
\end{equation}
and
\begin{equation}
\pi(dx)=\frac{1}{Z}\exp\{ -U(x) \}dx
\end{equation}
with $Z=\int_{\mathbb{R}}\exp\{ -U(y) \}dy<\infty$.

Since this note concerns the polynomial ergodicity of the one-dimensional Zig-Zag process, we now recall the definition of polynomial ergodicity.

\begin{definition}
Let $(Z_t)_{t \geq 0}$ be a Markov process with state space $E$, having invariant probability measure $\mu$ and let $k>0$. We say that the process is polynomially ergodic of order $k$ if there exists a function $M:E \rightarrow \mathbb{R}_{>0}$ such that for all $z \in E$ and $t \geq 0$
\begin{equation*}
\| \mathbb{P}_z\left( Z_t \in \cdot \right) - \mu (\cdot) \|_{TV}    \leq \frac{M(z)}{t^k},
\end{equation*}
where $\mathbb{P}_z$ denotes the law of the process starting from $z$.
\end{definition}

We will make the following assumption, typically verified in practice.

\begin{assumption}\label{as:1}
Assume that there exists an $\nu>0$ and a compact set $C \subset \mathbb{R}$ such that for all $x \notin C$, 
\begin{equation}\label{gradient.log.lik.growth.assump.poly:1}
|U'(x)| \geq \dfrac{1+\nu}{|x|}.
\end{equation}
\end{assumption}

\begin{remark}
This assumption directly implies that there exists a $c'$ such that for all $x \in \mathbb{R}$,
$U(x) \geq (1+\nu) \log(|x|)-c'$. This is an assumption made in \cite{bierkens.roberts.zitt:2019} in order to prove non-evanescence of the Zig-Zag process and it's a natural assumption given that the function $\exp\{ -U(x) \}$ must be integrable.
\end{remark}

\noindent We also need the following assumption for the refresh rate $\gamma$.
\begin{assumption}\label{as:2}
Assume that the refresh rates satisfy
\begin{equation}
\lim_{|x|\rightarrow \infty}\dfrac{\gamma(x)}{\left | U'(x) \right |}=0.
\end{equation}
\end{assumption}

\noindent 
Assumption \ref{as:2} ensures that the bouncing events will vastly outnumber the refresh events, at least in the tails of the target. Whilst we have not been able to establish the necessity of this assumption for the conclusions of Theorem \ref{prop.drift:1}, some control over 
$\dfrac{\gamma(x)}{\left | U'(x) \right |}$ is definitely needed.
Indeed in the regime where $\lim_{|x|\rightarrow \infty}\dfrac{\gamma(x)}{\left | U'(x) \right |}=+\infty$ (as would be the case with constant refresh rate and heavy-tailed target)
random direction changes would outnumber systematic ones, leading to
random walk/diffusive behaviour commonly associated with slow convergence.
In fact in this case, the algorithm would resemble a Random Walk Metropolis algorithm, which is known to converge at slower polynomial rate.
Therefore, we believe that some control of the refresh rate relative to $|U'|$ should be assumed for guarantees about specific rates of convergence to hold. This can be further supported by simulation studies. In Figure \ref{fif:mse.zz.t1} we present the Mean Square Error of estimating a tail probability
($\mathbb{P}(X \geq 5)$ for a standard Cauchy target)
using Zig-Zag with differing refresh rates.
We consider
$\gamma(x)=0$, $\gamma(x)=|U'(x)|$ and $\gamma(x)=1$. For each algorithm we generated 1000 independent realisations, all starting from $(-5,+1)$ and all realisations run until time $T=10^4$. For each time less than $T$ the average square error of the true probability (approximately equal to $0.0628$) is reported. It is clear that the smaller refresh rate leads to more rapid convergence.

\begin{figure}[ht]
\centering
\includegraphics[scale=1.5, width=12cm]{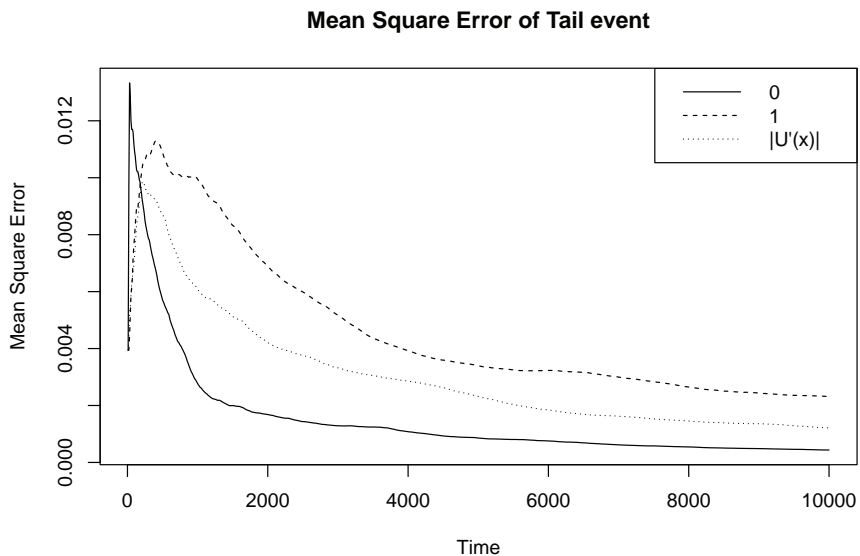}\label{fig:mse.zz.t1}
\caption{{\it Mean Square Error over time of three Zig-Zag algorithms with different refresh rates $\gamma(x)$, described in the upper right corner. The algorithms target the Cauchy distribution and are used to estimate the probability mass the Cauchy distribution assigns to the event $[5,+\infty)$. For all three algorithms, $1000$ independent realisations were generated and for each time less than $10^4$ the average square error between the approximation of the probability and the actual probability (approximately $0.0628$) are reported. Evidently, the fastest convergence is achieved by lowering the values of $\gamma$.}}\label{fif:mse.zz.t1}
\end{figure}

Our main result is the following.

\begin{theorem}[Polynomial ergodicity of Zig-Zag]
\label{prop.drift:1}
Suppose that $U$ satisfies Assumption \ref{as:1} and let $C$ and $\nu>0$ as in (\ref{gradient.log.lik.growth.assump.poly:1}). Suppose further that the refresh rate satisfies Assumption \ref{as:2}. 
Then, for any $k<\nu$, there exist constants $B,\delta>0$ and $\beta \in (0,1)$ such that if we let
\begin{equation}\label{as.drift:1}
V_{\beta, \delta}(x,\theta)=\exp \left \{ \beta U(x) + \delta \sgn(x) \theta \right \}
\end{equation}
then for all $(x,\theta) \in \mathbb{R} \times \{ -1,+1 \}$,
\begin{equation}\label{tv.bound:3}
\left \| \mathbb{P}_{x,\theta}(Z_t \in \cdot)-\mu(\cdot) \right \|_{TV} \leq \dfrac{B V_{\beta, \delta}(x,\theta)}{t^{1+k}}+\dfrac{B}{t^{k}},
\end{equation}
i.e. the process is polynomially ergodic of order $k$.
\end{theorem}

\begin{remark}\label{remark.2}
By carefully inspecting the proof of Theorem \ref{prop.drift:1}, we observe that Assumption \ref{as:2} can be weakened in the following sense. If $\nu$ is as in Assumption \ref{as:1} and if we fix $k<\nu$, then in order to prove polynomial ergodicity of order $k$, it suffices that there exists a small $\eta >0$ and that we ask that 
\begin{equation*}
    \lim_{|x| \rightarrow \infty}\frac{\gamma(x)}{|U'(x)|} \leq M
\end{equation*}
where 
\begin{equation}\label{M_k}
    M=M(k)=\left( \frac{k+1}{\nu +1}\left( 1+ \eta \right) -\eta \right)\frac{\left( 1- \frac{1+k}{1+\nu}\left( 1+\eta \right)  \right)^{1+\eta}}{1- \left( 1- \frac{1+k}{1+\nu}\left( 1+\eta \right) \right)^{1+\eta}}.
\end{equation}
We observe however that $M$ is not uniform in $k$. More precisely, assuming that $\eta$ is small enough so that $M(k)>0$ for all $k$, we have $\lim_{k \rightarrow \nu}M(k)=0$. A proof of this remark will be given in Section 3.
\end{remark}

An immediate corollary of Theorem \ref{prop.drift:1} is the following characterisation of the order of polynomial convergence of the one-dimensional Zig-Zag on Student distributions. For the following lower bound on the total variation distance from the invariant measure we use a type of argument similar to the proof of Theorem 2.1 in \cite{vasdekis.roberts.suzz:21}, suggested to us by Professor Anthony Lee in private communication.

\begin{corollary}\label{student.corollary:0}
Let $\nu>0$ and suppose $\pi$ is a Student distribution with $\nu$ degrees of freedom, i.e. 
\begin{equation}\label{student.equation:1}
    \pi(x) = \frac{1}{Z} \left ( 1+ \frac{x^2}{\nu} \right )^{-(\nu+1)/2}
\end{equation}
and the Zig-Zag process targets $\mu$ as in (\ref{def.mu:1}), having refresh rate that satisfies Assumption \ref{as:2}. For all $k<\nu$, there exist $\beta \in (0,1)$ and $\delta , B>0$ such that for all $(x,\theta) \in \mathbb{R} \times \{ -1,+1 \}$,
\begin{equation*}
\left \| \mathbb{P}_{x,\theta}(Z_t \in \cdot)-\mu(\cdot) \right \|_{TV} \leq \dfrac{BV_{\beta, \delta}(x)}{t^{1+k}}+\dfrac{B}{t^{k}},
\end{equation*}
i.e. for all $k<\nu$, the process is polynomially ergodic of order $k$. Furthermore, for any $k>\nu$, the process is not polynomially ergodic of order $k$. More specifically, there exists a constant $C'$ such that for all $t>0$ large enough 
\begin{equation*}
\left\|\mathbb{P}_{0,+1}(Z_t \in \cdot)-\mu(\cdot)\right\|_{TV}  \geq \frac{C'}{t^{\nu}}.
\end{equation*}
\end{corollary}

Corollary \ref{student.corollary:0} illustrates that Theorem \ref{prop.drift:1} is close to being tight. At least for Student distributions with $\nu$ degrees of freedom, it is established that the process is polynomially ergodic for order $k<\nu$ and it is not for any order $k>\nu$. Whether the process is polynomially ergodic of order $\nu$ is still an open question, but probably of limited practical importance.

\section{Proofs}\label{proofs:00}

Before we prove Theorem \ref{prop.drift:1} we recall the form of the strong generator of the Zig-Zag process.

\begin{definition}
We define $\mathcal{L}$ to be the operator acting on any $f \in C^1(E)$ so that for all $(x,\theta) \in E$,
\begin{equation}
    \mathcal{L}f(x,\theta)=\theta f'(x,\theta)+\left([\theta U'(x)]^++\gamma(x)\right)\left( f(x,-\theta)-f(x,\theta) \right) .
\end{equation}
It can be proven (see for example \cite{davis:18}) that $\mathcal{L}$ is the restriction on $C^1$ of the strong generator of $(Z_t)_{t \geq 0}$.
\end{definition}

 General conditions for polynomial ergodicity
 can be found in Theorems 3.2 and 3.4 of \cite{douc.fort.guillin:09} and Theorem 1.2 of \cite{bakry.cttiaux.guillin:08} (see also Corollary 6 in \cite{fort.roberts:05} for some earlier results on sub-geometric convergence).
Here we use a result found and proved in the presented form in the unpublished lecture notes \cite{hairer:16}.
\begin{theorem}[Hairer \cite{hairer:16} Theorem 4.1]\label{hairer:1}
Let $(X_t)_{t \geq 0}$ a continuous time Markov process on $X$ with strong generator $\mathcal{L}$. Suppose that there exists a function $V:X \rightarrow [1,+\infty)$ and a constant $K$ such that for all $x \in X$
\begin{equation}\label{drift:1}
\mathcal{L}V(x)\leq K- f(V)
\end{equation}
for a function $f:[0,+\infty)\rightarrow [0,+\infty)$ strictly concave, increasing, with $f(0)=0$, $\lim_{s \rightarrow +\infty}f(s)=+\infty$. Suppose further that all the sub-level sets of $V$ are pre-compact and small. Then, the following hold:
\begin{enumerate}
\item There exists a unique invariant measure $\mu$ for the process such that\\ $\int f(V(x)) \mu(dx)<\infty$.
\item Let $H_{f}(u)=\int_{1}^u1/f(s)ds$, then there exists a constant $B>0$ such that for every $x \in X$
\begin{equation}\label{tv.bound:1}
\| \mathbb{P}_x(X_t \in \cdot)-\mu(\cdot) \|_{TV} \leq \dfrac{B V(x)}{H^{-1}_{f}(t)}+\dfrac{B}{f \circ H^{-1}_{f}(t)}.
\end{equation}
\end{enumerate}
\end{theorem}

\begin{proof}[Proof of Theorem \ref{prop.drift:1}]
Suppose that $U$ satisfies Assumption \ref{as:1} and let $k<\nu$. Select $a$ such that $k=a/(1-a)$ so that $a<1-1/(1+\nu)$. For any $\tilde{\beta} \in (0,1)$ we have that there exists a $c_0>0$ such that for all $x \notin C$
\begin{align*}
\left( V_{\tilde{\beta}, \delta}(x,\theta) \right)^{1-a}|U'(x)|\geq c_0 \exp \{ \tilde{\beta} (1-a)(1+\nu) \log |x| \}\dfrac{1+\nu}{|x|}=c_0(1+\nu) \left | x \right |^{\tilde{\beta} (1-a) (1+\nu)-1}
\end{align*}
Since $(1-a)(1+\nu)-1>0$, there exists a $\beta$ close to $1$ such that $\beta (1-a) (1+\nu)-1>0$, so 
\begin{equation}\label{poly.ergo.proof.eq:1}
\lim_{|x| \rightarrow \infty}V^{1-a}_{\beta}(x,\theta)|U'(x)|=+\infty.
\end{equation}

\noindent Now, $V_{\beta, \delta} \in C^1$ and $\lim_{|x|\rightarrow \infty}V_{\beta, \delta}(x,\theta)=+\infty$ so all the level sets are compact. Since the process is positive Harris recurrent and some skeleton is irreducible (see \cite{bierkens.roberts.zitt:2019}) we get from Proposition 6.1 of \cite{meyn.tweedie.2:93} that the level sets are also small. Since $\lim_{|x|\rightarrow \infty}U(x)=+\infty$ it is clear that $V_{\beta, \delta}$ is bounded below away from $0$ so by multiplying with an appropriate constant we can assume that $V_{\beta, \delta}(x,\theta)\geq 1$ for all $(x,\theta)$.
We calculate
\begin{equation*}
\mathcal{L}V_{\beta, \delta}(x,\theta)=V_{\beta, \delta}(x,\theta)\left( \theta \beta U'(x)+([\theta U'(x)]^++\gamma(x))\left( \exp \{ -2 \theta \sgn(x) \delta \}-1 \right) \right).
\end{equation*}
 Note that due to Assumption \ref{as:1} and that $U(x)\xrightarrow{|x| \rightarrow \infty} +\infty$, we get that there exists a compact set $\tilde{C}$ such that for all $x \notin \tilde{C}$ $\sgn(U'(x))=\sgn(x)$. Therefore, when $x \notin \tilde{C}$ and $\theta \sgn(x)>0$,
\begin{align*}
\dfrac{\mathcal{L}V_{\beta, \delta}(x,\theta)}{V_{\beta, \delta}^a(x,\theta)}\leq V_{\beta, \delta}^{1-a}(x,\theta)|U'(x)|\left[ \beta+\left( \dfrac{\gamma(x)}{|U'(x)|}+1 \right) \left( \exp \{ -2 \delta \}-1 \right)\right]
\end{align*}
and when $\theta \sgn(x) <0$ 
\begin{equation*}
\dfrac{\mathcal{L}V_{\beta, \delta}(x,\theta)}{V_{\beta, \delta}^a(x,\theta)}\leq V_{\beta, \delta}^{1-a}(x,\theta)|U'(x)|\left[ -\beta +\dfrac{\gamma(x)}{|U'(x)|}(\exp \{ 2 \delta \} -1) \right].
\end{equation*}
Overall we have for $x \notin \tilde{C}$
\begin{equation}\label{gen.upper.bound:1}
\dfrac{\mathcal{L}V_{\beta, \delta}(x,\theta)}{V_{\beta, \delta}^a(x,\theta)}\leq V_{\beta, \delta}^{1-a}(x,\theta)|U'(x)| \max \left\{ \beta+\left( \dfrac{\gamma(x)}{|U'(x)|}+1 \right) \left( \exp \{ -2 \delta \}-1 \right) , \left[ -\beta +\dfrac{\gamma(x)}{|U'(x)|}(\exp \{ 2 \delta \} -1) \right] \right\}.
\end{equation}
Recall that $\dfrac{\gamma(x)}{|U'(x)|}\xrightarrow{|x| \rightarrow \infty}0$. Fix $\delta>-1/2\log(1-\beta)$ and by possibly increasing $\tilde{C}$, we have that there exists a constant $c'>0$ such that for all $x \notin \tilde{C}$ 
\begin{equation*}
\max \left\{ \beta+\left(\dfrac{\gamma(x)}{|U'(x)|}+1 \right) \left( \exp \{ -2 \delta \}-1\right), -\beta +\dfrac{\gamma(x)}{|U'(x)|}\left(\exp \{ 2 \delta \} -1 \right) \right\}<-c'<0.
\end{equation*}
Combining this with (\ref{poly.ergo.proof.eq:1}) and (\ref{gen.upper.bound:1})
we get that $V_{\beta, \delta}$ satisfies (\ref{drift:1}) with $f(u)=cu^a$ for $c=c'/2$ and with $K$ being an appropriate constant that bounds the continuous function $\mathcal{L}V_{\beta, \delta}+ f(V_{\beta, \delta})$ inside $\tilde{C}$.

\noindent Therefore, all the conditions of Theorem \ref{hairer:1} are satisfied. Note that
$H_{f}(s)=\int_1^s1/f(u)du=c^{-1}\int_1^su^{-a}du=\dfrac{1}{c(1-a)}(s^{1-a}-1)$ so  
\begin{equation*}
H^{-1}_{f}(t)=\left ( 1+c(1-a)t \right)^{1/(1-a)}
\end{equation*}
and therefore
\begin{equation*}
f \circ H^{-1}_{f}(t)= c \left ( 1+c(1-a)t \right )^{a/(1-a)}
\end{equation*}
Since we picked $a$ such that $k=a/(1-a)$, meaning that $k+1=1/(1-a)$, (\ref{tv.bound:3}) follows. 
\end{proof}

\begin{proof}[Proof of Remark \ref{remark.2}]
Fix $k<\nu$ and set $a$ such that $k=a/(1-a)$, therefore $1-a=1/(1+k)$. Our goal is to find an appropriate upper bound on $\frac{\gamma(x)}{|U'(x)|}$ that will guarantee that the RHS of (\ref{gen.upper.bound:1}) is negative for appropriately chosen $\beta$ and $\delta$. For this it suffices to prove that for some $\beta,\delta >0$, the maximum  appearing in the RHS of (\ref{gen.upper.bound:1}) is negative, bounded away from zero and that (\ref{poly.ergo.proof.eq:1}) holds. These two conditions will guarantee that the drift condition (\ref{drift:1}) holds for $V_{\beta , \delta}$ and we can conclude similar to the proof of Theorem \ref{prop.drift:1}.
Since from the proof of that theorem
\begin{align*}
\left( V_{\tilde{\beta}, \delta}(x,\theta) \right)^{1-a}|U'(x)|\geq c_0(1+\nu) \left | x \right |^{\tilde{\beta} (1-a) (1+\nu)-1},
\end{align*}
in order to guarantee (\ref{poly.ergo.proof.eq:1}), it suffices to pick $\beta=(1+\eta)(1+k)/(1+\nu)$ for some $\eta >0$. From the discussion after (\ref{gen.upper.bound:1}), we can pick $\delta=-\frac{1}{2}(1+\eta)\log(1-\beta)=-\frac{1}{2}(1+\eta)\log\left(1-(1+\eta)(1+k)/(1+\nu)\right)$. With this choice of $\beta, \delta$ the first part of the maximum of the RHS in (\ref{gen.upper.bound:1}) will be negative, bounded away from zero. The second part of the maximum writes
\begin{equation*}
-\beta +\dfrac{\gamma(x)}{|U'(x)|}(\exp \{ 2 \delta \} -1)=-(1+\eta)\frac{1+k}{1+\nu}+\dfrac{\gamma(x)}{|U'(x)|}\left( \left(\frac{1}{(1-(1+\eta)(1+k)/(1+\nu))}\right)^{(1+\eta)}  -1\right).
\end{equation*}
Therefore, solving the inequality $-\beta +\dfrac{\gamma(x)}{|U'(x)|}(\exp \{ 2 \delta \} -1)<-\eta$ (which would guarantee that the maximum on the RHS of (\ref{gen.upper.bound:1}) is negative, bounded away from zero), we get
\begin{equation*}
    \frac{\gamma(x)}{|U'(x)|} \leq M(k)
\end{equation*}
where $M(k)$ as in (\ref{M_k}). 
\end{proof}

\begin{proof}[Proof of Corollary \ref{student.corollary:0}]
Let $\pi$ as in (\ref{student.equation:1}). Then for all $\delta'>0$ there exists a compact set $C$ with $0 \in C$, such that for all $x \notin C$
\begin{align*}
|U'(x)|=\dfrac{(\nu+1)|x|}{\nu+x^2}\geq \dfrac{1+\nu-\delta'}{|x|}
\end{align*}
Therefore for every $\delta'>0$, the distribution satisfies Assumption \ref{as:1} where the $\nu$ in that assumption is equal to $\nu-\delta'$. From Proposition \ref{prop.drift:1}, for all $k<\nu$, there exists  a $\beta \in (0,1), \delta >0$ and $B>0$ such that for all $(x,\theta) \in \mathbb{R} \times \{ -1,+1 \}$,
\begin{equation*}
\left \| \mathbb{P}_{x,\theta}(Z_t \in \cdot)-\mu(\cdot) \right \|_{TV} \leq \dfrac{B V_{\beta, \delta}(x)}{t^{1+k}}+\dfrac{B}{t^{k}}.
\end{equation*}
Now, suppose that the Zig-Zag starts from $x=0$, $\theta=+1$. There exists a $C_0$ and $K>0$ such that for all $|x| \geq K$, $\pi(x)\geq C_0 |x|^{-\nu-1}$. Fix a time $t>K$. Let $A_t= \{ x: x>t \}$. After less or equal to time $t$ has passed, the Zig-Zag will not have hit $A_t$.  We therefore get for all $t > K$,
\begin{align*}
\|\mathbb{P}_{0,+1}(Z_t \in \cdot)-\mu(\cdot)\|_{TV} \geq & 
\left | \mathbb{P}_{0,+1}(X_t \in A_t^c)-\pi(A_t^c) \right |=\pi(A_t^c)= \int_{|x|>t} \pi(x)dx \geq \\
& \geq 2C_0\int_t^{+\infty}x^{-\nu-1}dx=\dfrac{C_0}{\nu}\frac{1}{t^{\nu}}. 
\end{align*}
\end{proof}

\section{Discussion}
\noindent It is interesting to compare the Zig-Zag polynomial convergence rates with the ones of the one-dimensional Random Walk Metropolis (RWM) and the Metropolis-adjusted Langevin algorithm (MALA) algorithms. In fact, it is shown in \cite{jarner.tweedie:03}, Propositions 4.1 and 4.3 (see also \cite{jarner.roberts:07}) that when targeting a Student distribution with $\nu$ degrees of freedom with any finite variance proposal RWM or with MALA, one gets polynomial order of convergence $(\nu/2)^-$, i.e. for any $\epsilon>0$, the polynomial rate of convergence is at least $\nu/2-\epsilon$. It is however proven not to be $\nu/2$. In \cite{jarner.roberts:07} the authors provide these lower bounds for the convergence rates, while in \cite{jarner.tweedie:03} they provide the upper bound. As proven in this note, the one-dimensional Zig-Zag has polynomial rate of convergence $\nu^-$ in the same setting, which is better than RWM or MALA. This phenomenon was also observed in simulations in \cite{bierkens.duncan:17}. We conjecture that the advantage of the Zig-Zag is due its momentum, which, in a one-dimensional, uni-modal setting with zero refresh rate, will force the process to never switch direction before it hits the mode. This diminishes any possible diffusive behaviour of the process at the tails and helps the algorithm converge faster. We should note here that better polynomial rates (and more precisely, arbitrarily better rates) can be achieved for the Random Walk Metropolis (RWM) if one introduces a proposal with heavier tails. However, a natural analogue of this modification is to allow the Zig-Zag to speed up and move faster in areas of lower density. This idea is further discussed in \cite{vasdekis.roberts.suzz:21} and proven to be able to provide exponentially ergodic algorithms even on heavy tailed targets, which can outperform the simple Zig-Zag in the sense of having better effective sample size per number of likelihood evaluations.





\acks 
\noindent The authors would like to acknowledge Professor Anthony Lee for an indication of the proof of the lower bound of the total variation distance in Corollary 2.5. They would also like to acknowledge Andrea Bertazzi, George Deligiannidis and Krzysztof Latuszy\'{n}ski for helpful discussions. Furthermore, they would like to thank the anonymous referee and the associated editor for the helpful comments that improved this note.

\fund 
\noindent 
G. Vasdekis was supported by the EPSRC as part of the MASDOC DTC (EP/HO23364/1) and the Department of Statistics at the University of Warwick (EP/N509796/1). He is also supported by NERC (NE/T00973X/1) under Dr. Richard Everitt. G. O. Roberts was supported by EPSRC under the CoSInES (EP/R018561/1) and Bayes for Health (EP/R034710/1) programmes.

\competing 
\noindent 
There were no competing interests to declare which arose during the preparation or publication process of this article.



%
%
%
%
\bibliographystyle{APT} 
\bibliography{Biblio}       






\end{document}